\documentclass{amsart}

\advance\oddsidemargin by -1.0cm
\advance\evensidemargin by -1.0cm
\theoremstyle{plain}
\newtheorem{corollary}{Corollary}[section]
\newtheorem{lemma}{Lemma}
\newtheorem{theorem}{Theorem}

\theoremstyle{definition}

\newtheorem{remark}{Remark}
\newtheorem{definition}{Definition}

\begin{document}


%
%

\title{Quaternionic Heisenberg groups as naturally reductive homogeneous 
spaces}

\author{Ilka Agricola}

\address{Fachbereich Mathematik und Informatik, Philipps-Universit\"at Marburg, 
Hans-Meerwein-Stra\ss e, D-$35032$ Marburg, Germany\\
{\normalfont\ttfamily agricola@mathematik.uni-marburg.de } }

\author{Ana Cristina Ferreira}

\address{Centro de Matem\'atica, Universidade do Minho, Campus de Gualtar, 
$4710$-$057$ Braga, Portugal\\
{\normalfont\ttfamily  anaferreira@math.uminho.pt} }

\author{Reinier Storm}

\address{Fachbereich Mathematik und Informatik, Philipps-Universit\"at Marburg, 
Hans-Meerwein-Stra\ss e, D-$35032 Marburg$, Germany\\
{\normalfont\ttfamily stormrw@mathematik.uni-marburg.de} }


\maketitle


\begin{abstract}
In this note, we describe the geometry of the quaternionic Heisenberg groups from 
a Riemannian viewpoint. We show, in all dimensions, that they carry an almost 
$3$-contact metric structure which allows us to define the metric connection that 
equips these groups with the structure of a naturally reductive homogeneous space.
It turns out that this connection, which we shall call the canonical connection 
because of its analogy to the $3$-Sasaki case, preserves the horizontal and vertical
distributions and even the quaternionic contact structure of the quaternionic 
Heisenberg groups.
We focus on the $7$-dimensional case and prove that the canonical connection can also be 
obtained by means of a cocalibrated $G_2$ structure. We then study the spinorial 
properties of this group and present the noteworthy fact that it is the only
known  example of a manifold which carries generalized Killing spinors 
with three different eigenvalues.   
\end{abstract}
%
%
\newcommand{\m}{\ensuremath{\mathfrak{m}}}
\newcommand{\ox}{\otimes}
\renewcommand{\k}{\ensuremath{\mathfrak{k}}}
\newcommand{\g}{\ensuremath{\mathfrak{g}}}
\newcommand{\kr}{\ensuremath{\mathcal{R}}}
\def\haken{\mathbin{\hbox to 6pt{%
                 \vrule height0.4pt width5pt depth0pt
                 \kern-.4pt
                 \vrule height6pt width0.4pt depth0pt\hss}}}

\section{Introduction}	
Among all homogenous Riemannian manifolds, naturally reductive spaces are a class 
of particular
interest. Traditionally, they are defined as  Riemannian manifolds $(M=G/K,g)$ with a
reductive complement $\m$ of $\k$ in $\g$ such that
\begin{equation}\label{eq.NR}
\langle [X,Y]_\m, Z\rangle + \langle Y, [X,Z]_\m\rangle\ =\ 0 \ 
\text{ for all } X,Y,Z\in\m,
\end{equation}
where $\langle-,-\rangle$ denotes the inner product on $\m$ induced from 
$g$. For any reductive homogeneous space, the submersion $G\rightarrow G/K$ induces a 
connection that is called the \emph{canonical connection}.  It is   a
metric connection $\nabla$ with torsion $T(X,Y)=-[X,Y]_\m$ wich satisfies 
$\nabla T = \nabla\kr = 0$, and  condition (\ref{eq.NR}) thus states that
a naturally reductive homogeneous space is a reductive space for which the torsion
 $T(X,Y,Z):=g(T(X,Y),Z)$ (viewed as a $(3,0)$-tensor) is a $3$-form on $G/K$ (see
\cite[Ch.\,X]{Kobayashi&N2} as a general reference). 
Classical examples of naturally reductive homogeneous spaces
include irreducible symmetric spaces, isotropy irreducible 
homogeneous manifolds, Lie groups with a biinvariant metric, and
Riemannian 3-symmetric spaces.

In the recent article \cite{AFF14}, the first two authors together with
Thomas Friedrich (Berlin) initiated a systematic investigation and, up to dimension
six,
achieved the classification of naturally reductive homogeneous spaces. This is done
by applying recent results and techniques from
the holonomy theory of metric connections with skew torsion. 
\begin{definition}
We call a Riemannian manifold
$(M,g)$ \emph{naturally reductive} if it is a homogeneous space $M=G/K$
endowed with a metric connection $\nabla$ with skew torsion $T$ such that its
torsion and curvature $\kr$ are $\nabla$-parallel, 
i.\,e.~$\nabla T=\nabla \kr=0$.
\end{definition}
If $M$ is connected, complete, and simply connected, a result
of Tricerri asserts that the space is indeed naturally reductive in
the traditional sense \cite{Tricerri93}.


\subsection*{Acknowledgments}
One of the referees of this paper suggested a comparison
of the canonical  connection (constructed in Theorem \ref{thm.connection}) with the 
Biquard connection preserving the quaternionic contact (qc) structure of the 
quaternionic Heisenberg group. We thank Ivan Minchev
(Brno) for very valuable comments on this topic; 
the results are described in Section \ref{sec.qc}.

Ilka Agricola and Ana Ferreira 
acknowledge financial support by the
DFG within the priority programme 1388 "Representation theory".
Ana Ferreira thanks Philipps-Universit\"at Marburg for its
hospitality during a research stay in May-July 2013 and October 2014, and she also 
acknowledges partial financial support by the FCT through the 
project PTDC/MAT/118682/2010 and the University of 
Minho through the FCT projects PEst-C/MAT/UI0013/2011 and PEst-OE/MAT/UI0013/2014. 

\section{Geometry of quaternionic Heisenberg groups}
%
\subsection{Lie groups of type $H$}
%
Lie groups of type $H$ or generalized Heisenberg groups are, as the
name indicates, a generalization of the classical Heisenberg groups. We will give a 
brief overview
of such groups and treat the naturally reductive ones from the point of view of contact 
geometry.
We start by describing the Lie algebra of such a group. Let $Z$ and $V$ be 
two (real) vector spaces
of any positive dimension. Equip such vector spaces with some inner product, which
shall be denoted for both spaces as $\langle -  ,  - \rangle$. Suppose there is a linear 
map $k: Z \longrightarrow \mathrm{End}(V)$ such that
$$\Vert k(a)x \Vert = \Vert x \Vert \Vert a \Vert \quad \mbox{ and } 
\quad k(a)^2 = -\Vert a \Vert^2 \mathrm{Id} $$
for $a \in Z$, $x, y \in V$. We can use the map $k$ to define the Lie algebra $\mathfrak{n}$ as 
the direct sum $\mathfrak{n} = Z\oplus V$ together with the bracket defined by
$$[a + x, b + y] = [x, y] \quad \mbox { and } \quad \langle [x, y], a \rangle 
= \langle k(a)x, y \rangle,$$
where $a, b \in Z$ and $x, y \in V$. Then $\mathfrak{n}$ is said to be a 
Lie algebra of type $H$. Remark that $\mathfrak{n}$ is
a 2-step nilpotent Lie algebra with center $Z$. There are infinitely many 
Lie algebras of type $H$ with center of any given dimension.  The connected, 
simply-connected Lie group
$N$ with Lie algebra $\mathfrak{n}$ is said to be a Lie group of type $H$. 
Observe also that the 
Lie algebra
$\mathfrak{n}$ can be equipped with an inner product such that the 
decomposition $Z \oplus V$ 
is orthogonal and $N$ is
endowed with a left invariant metric $g$ induced by the inner product on $\mathfrak{n}$. 
Of particular interest
are the Lie algebras which are obtained from the composition algebras $W$ -- the 
complex numbers $\mathbb{C}$,
the quaternions $\mathbb{H}$ and the octonions $\mathbb{O}$ as follows: $Z$ is 
the space formed by purely imaginary
numbers, $V$ is a power of $W$, i.e., $V = W^n$ and $k: Z \longrightarrow \mathrm{End}(V)$ 
is simply the linear map given
by ordinary scalar multiplication. The corresponding groups are the Heisenberg groups or their
quaternionic or octonionic analogs.
As far as naturally reductive spaces go, we have the following theorem of Tricerri and Vanhecke.
\begin{theorem}[\cite{TV83}, Theorem 9.1, page 96]
The Lie group $N$ with its left invariant metric $g$ is naturally reductive if 
and only if $N$ is a Heisenberg group or a quaternionic Heisenberg group.
\end{theorem}
Heisenberg groups of dimension $2n+1$ were described as naturally reductive spaces
in \cite{AFF14}. The Heisenberg group of dimension $5$ was the first known example of a
manifold with parallel skew torsion carrying a
Killing spinor with torsion that does not admit a Riemannian Killing 
spinor \cite{ABBK13}, \cite{BB12}.
\subsection{The quaternionic Heisenberg group $N_p$ of dimension $4p+3$}
%
Let $p\in \mathbb{N}$, $V$ be the space of quaternions
and $Z$ be the space of imaginary quaternions. Consider the Lie algebra 
$\mathfrak{n}_p = Z \oplus V^p$, of dimension $4p+3$, and denote
by $N_p$ its corresponding connected, simply connected Lie group. For  ease of notation
we will denote by $z_1, z_2, z_3$ the standard elements $i, j, k$ in $Z$ and by 
$\tau_r$, $\tau_{p+r},\  \tau_{2p+r},\ \tau_{3p+r}$ the elements $1, i, j, k$ in each copy of $V^p$, 
respectively,  $r = 1, \dots, p$. More concisely, we make the following identifications
for $r = 1, \dots, p$
$$
\tau_r \longmapsto 1, \quad  
z_1,\ \tau_{p+r} \longmapsto i, \quad   
z_2,\ \tau_{2p+r} \longmapsto j,\quad
z_3,\  \tau_{3p+r}\longmapsto k.
$$
%
%
We introduce a parameter $\lambda$ in our metric by declaring that the set
$\xi_i := \frac{z_i}{\lambda}$ ($1\leq i\leq 3$), $\tau_l$ ($1\leq l\leq 4p)$
is an orthornormal frame for the metric $g_\lambda$, $\lambda > 0$. The commutator relations are 
now written as
$$\begin{array}{lll}
\left[\tau_r, \tau_{p+r}\right] = \lambda\, \xi_1 \qquad &  [\tau_r, \tau_{2p+r}] 
= \lambda\, \xi_2 \qquad & [\tau_r, \tau_{3p+r}] = \lambda\, \xi_3\\
\left[\tau_{2p+r} , \tau_{3p+r}\right] = \lambda\, \xi_1 & [\tau_{3p+r}, \tau_{p+r}] 
= \lambda\, \xi_2 & [\tau_{p+r},\tau_{2p+r}] = \lambda\, \xi_3 
\end{array}$$
with all the remaining commutators begin zero.  The Levi-Civita connection can be computed, but
it is not very insightful, so we will not reproduce this calculation
here. Let us just point out that we have three Riemannian Killing fields, 
namely $\xi_1, \xi_2, \xi_3$. We
do not have a distinguished direction, but a distinguished $3$-dimensional
distribution in the tangent bundle, so the best approach to study the geometry of these 
groups is to consider
3-contact structures. Let $\eta_i$ be the dual form of $\xi_i$, i = 1, 2, 3, and 
$\theta_l$ 
be the dual form of $\tau_l$, respectively, $l = 1, \dots ,4p$. 
Define the $(1,1)$-tensors
$$
\begin{array}{l}
\varphi_1 = \eta_2 \otimes \xi_3 - \eta_3 \otimes \xi_2 + \displaystyle\sum_{r=1}^p 
[\theta_r\otimes \tau_{p+r} - \theta_{p+r} \otimes \tau_r + \theta_{2p+r} \otimes \tau_{3p+r} - \theta_{3p+r} \otimes \tau_{2p+r}],\\
\varphi_2 = \eta_3 \otimes \xi_1 - \eta_1 \otimes \xi_3  + \displaystyle\sum_{r=1}^p 
[\theta_r\otimes \tau_{2p+r} 
- \theta_{2p+r} \otimes \tau_r + \theta_{3p+r} \otimes \tau_{p+r} - \theta_r \otimes \tau_{3p+r}],\\
\varphi_3 = \eta_1 \otimes \xi_2 - \eta_2 \otimes \xi_1 + \displaystyle\sum_{r=1}^p 
[\theta_r\otimes \tau_{3p+r} 
- \theta_{3p+r} \otimes \tau_r + \theta_{p+r} \otimes \tau_{2p+r} - \theta_{2p+r} \otimes \tau_{p+r}]. 
\end{array}
$$
It is easy to check that the triple $(\varphi_1, \varphi_2, \varphi_3)$ 
satisfies the compatibility equations 
\begin{displaymath}
\varphi_i  \ =\  \varphi_j \varphi_k - \eta_k \otimes \xi_j 
\ =\  - \varphi_k \varphi_j + \eta_j \otimes \xi_k,
\end{displaymath}
(for $(i,j,k)=(1,2,3)$ and cyclic
permutations) and also that all three almost contact structures are compatible with the metric. 
All in all, $(N_p, \varphi_i, \xi_i,g_\lambda)$ is an almost $3$-contact
metric manifold (see the classical monography \cite{Blair} for more information 
on this topic). Note that
none of the structures $\varphi_i$ is quasi-Sasaki, so $N_p$ is not a $3$-(quasi)-Sasaki
manifold, but all three are normal (vanishing Nijenhuis tensor). The vertical subbundle 
$T^v$ is spanned by $\xi_1,\xi_2,\xi_3$, the horizontal
subbundle $T^h$ is its orthogonal complement.
For later use, let us write down the formulas for the differentials $d\eta_i$:
\begin{equation}\label{eq.d-eta}
d\eta_i \ =\ -\lambda\sum_{r=1}^p [\theta_{r,ip+r}+ \theta_{(i+1)p+r,(i+2)p+r}], \qquad
i=1,2,3.
\end{equation}
Henceforth, we write $\theta_{ij}$ for $\theta_i\wedge\theta_j$ (and similarly 
for $\eta$), and the
index $i$ is understood modulo $3$, i.\,e.~for $i=2$, $i+2=1 \bmod 3$, thus
$(i+2)p+r$ is to be read as $p+r$. In particular, for $p=1$, we have the simple formulas
\begin{equation}\label{eq.d-eta2}
d\eta_1\ =\ -\lambda(\theta_{12}+ \theta_{34}),\quad
d\eta_2\ =\ -\lambda(\theta_{13}- \theta_{24}),\quad
d\eta_3\ =\ -\lambda(\theta_{14}+ \theta_{23}).
\end{equation}
Each of the three almost contact structures $(\varphi_i,\xi_i) \ (i=1,2,3)$ 
of $N_p$ has a characteristic connection \cite{FI02}, but
they do not coincide, and each of them is not well adapted to the underlying 
$3$-contact structure.
In the article \cite{AF10}, a notion of canonical connection was proposed
for $7$-dimensional $3$-Sasaki manifolds. This was the metric connection 
$\nabla$ with skew
torsion $T= \sum_{i=1}^3 \eta_i\wedge d\eta_i$; it was shown to preserve the
vertical and horizontal subbundles, and to admit a $\nabla$-parallel 
spinor field $\psi$ with
the property that the fields $\xi_i\cdot \psi$ were the Riemannian Killing spinors
of the manifold. The construction was done by using an intermediate 
cocalibrated $G_2$-structure.

Even though we are not in the 3-Sasaki case, we will now show that a similar 
connection can be constructed on $N_p$, and that this connection gives $N_p$ a 
naturally reductive homogeneous structure.
The special case $p = 1$ will be considered separately in Section \ref{sec.n=7}.
\begin{theorem}\label{thm.connection}
On the almost $3$-contact metric manifold $(N_p, \varphi_i, \xi_i,g_\lambda)$, 
the metric connection
$\nabla$ with skew torsion
\begin{equation}
T \ =\ \eta_1\wedge d\eta_1 + \eta_2 \wedge d\eta_2 + \eta_3 \wedge d\eta_3 
- 4\lambda \eta_{123} \label{Eq:3-form}
\end{equation}
has the following properties:
\begin{enumerate}
\item Its torsion and curvature are $\nabla$-parallel, $\nabla T= \nabla \kr =0$;
\item Its holonomy algebra is isomorphic to $\mathfrak{su}(2)$, acting irreducibly
on $T^v$ and on $T^h$ by $p$ copies of its $4$-dimensional representation;
in particular, $\nabla$ preserves the vertical and horizontal subbundles $T^v, T^h$.
\end{enumerate}
\end{theorem}
\begin{proof}
%
One computes that
the metric connection with skew torsion $T$ is described by the map 
$\Omega: \mathfrak{n}_p\longrightarrow \Lambda^2 \mathfrak{n}_p = 
\mathfrak{so}(\mathfrak{n}_p)$ 
$$
\Omega(\xi_i) =  -\lambda H_i, \qquad   \Omega(\tau_l) =  0, \qquad i=1,2,3, \ l=1,\ldots,4p
$$
where $H_1$, $H_2$, $H_3$ are given by
$$
H_1= -\frac{1}{\lambda} d\eta_1 + 2 \eta_2 \wedge \eta_3,\quad
 H_2=-\frac{1}{\lambda} d\eta_2 - 2 \eta_1 \wedge \eta_3,\quad
  H_3= - \frac{1}{\lambda} d\eta_1 + 2 \eta_1 \wedge \eta_2.
$$
The elements $H_1, H_2, H_3$ satisfy the commutator relation of $\mathfrak{su}(2)$, 
that is, $[H_1, H_2]= 2 H_3$, $[H_3,H_1]= 2 H_2$, $[H_2,H_3] = 2 H_1$. This fact yields 
that the curvature tensor can be readily computed to be 
$$
\mathcal{R}= \lambda^2 [H_1\otimes H_1 + H_2 \otimes H_2 + H_3 \otimes H_3]
$$
and the holonomy algebra $\mathfrak{h}$ of our connection is $\mathfrak{su}(2).$ 
Clearly, 
both $T$ and $\mathcal{R}$ are $\mathfrak{h}$-invariant, so we can conclude that both 
$T$ and $\mathcal{R}$ are parallel objects. It is then established that every 
quaternionic 
Heisenberg group has the structure of a naturally reductive homogeneous space.  

The elements $H_j$ of the holonomy algebra $\mathfrak{h}$  act on a vector field
$\xi_i$ or $\tau_l$ by inner product, i.\,e.~$H_j\cdot\xi_i=\eta_i\haken H_j $
and $H_j\cdot\tau_l=\theta_l\haken H_j$. Thus, the explicit formulas
(\ref{eq.d-eta}) for $d\eta_i$ and therefore $H_i$ imply that $\mathfrak{h}$ 
acts irreducibly on $T^v=\mathrm{Span}(\xi_1,\xi_2,\xi_3)$ and leaves invariant
the space $\mathrm{Span}(\theta_r, \theta_{p+r},\theta_{2p+r},\theta_{3p+r})$
for each $r=1\ldots, p$. In particular, this means that not only $T$ is $\nabla$-parallel,
but also $\eta_{123}$ and each of the $\theta_r\wedge\theta_{r,p+r}\wedge\theta_{2p+r}\wedge\theta_{3p+r}$
as well.
\end{proof}
\begin{definition}
The connection $\nabla$ described in the previous theorem will be called the
\emph{canonical connection} of the almost $3$-contact metric manifold 
$(N_p, \varphi_i, \xi_i,g_\lambda)$.
\end{definition}
\begin{remark}
It is interesting to observe that  $\nabla$-Ricci curvature is a diagonal matrix
$$
\mathrm{Ric}^{\nabla} = \mathrm{diag}(-8\lambda^2, -8\lambda^2, 
-8\lambda^2, -3\lambda^2,\ldots, -3\lambda^2),
$$
even though it is never a multiple of the identity. We can also deduce the 
$\nabla$-scalar curvature and the Riemannian scalar curvature to be negative, 
more precisely, $s^\nabla = -12 \lambda^2 (p+2)$ and
$s^g = s^\nabla + (3/2) \|T\|^2 = - 3p\lambda^2$.
\end{remark}
\subsection{Compatibility of $\nabla$ with the qc structure and comparison to the
 Biquard connection}\label{sec.qc}
%
The quaternionic Heisenberg group is the standard example of a non compact
quaternionic contact (qc for short) manifold. Standard references are 
\cite{Biquard00} or \cite{Duchemin06}, we shall mainly follow the notations and
definitions (which vary slightly) from \cite{Ivanov&M&V14}. 
Instead of giving the abstract  definition of 
a qc structure, we quickly define all relevant quantities for the quaternionic 
Heisenberg group and check the required properties:
\begin{enumerate}
\item The endomorphism fields $I_i:=\varphi_i\big|_{T^h}$ are complex
structures on $T^h$ satisfying the commutation relations
of the quaternions, $I_1 I_2 I_3= -\mathrm{Id}$;
\item the triple of $1$-forms $\tilde\eta_i:=-\frac{2}{\lambda}\eta_i$
satisfies $T^h= \cap_i \ker \tilde\eta_i$;
\item the differentials of the  $1$-forms $\tilde\eta_i$ and the vector fields
$\tilde\xi_j:= -\frac{1}{2}z_j$ satisfy the identities
$d\tilde\eta_j(X,Y)= 2 \, g(I_jX,Y)$ for $X,Y\in T^h$ and
$d\tilde\eta_j(\tilde\xi_k,X)= - d\tilde\eta_k(\tilde\xi_j,X)$ for 
$X\in T^h, \ j,k=1,2,3$.
\end{enumerate}
In general, the condition for a metric connection $\nabla$ to preserve the qc  
structure reduces to the requirement that $\nabla$  preserves the splitting 
$T^h \oplus T^v$ and has the additional two properties
\begin{enumerate}
\item $\nabla(I_1\ox I_1+ I_2\ox I_2+ I_3\ox I_3 )=0$;
\item $\nabla (\tilde\xi_1\ox I_1+\tilde\xi_2\ox I_2+\tilde\xi_3\ox I_3)=0$.
\end{enumerate}
From the formulas given in Theorem \ref{thm.connection} and its proof, it is clear
 that the canonical connection $\nabla$ satisfies these condition, hence we conclude:
\begin{corollary}
The canonical connection $\nabla$ preserves the above defined qc structure of the
quaternionic Heisenberg group $N_p$.
\end{corollary}
In fact, one can show that $\nabla$ is the unique metric connection preserving 
the qc structure whose torsion is skew-symmetric. The Biquard connection is the
most commonly used  qc connection; on $N_p$, it is given by the trivial map
$\Omega^B: \mathfrak{n}_p\longrightarrow \Lambda^2 \mathfrak{n}_p = 
\mathfrak{so}(\mathfrak{n}_p),\ \Omega^B=0$. Obviously, it's flat and has 
vanishing holonomy. The canonical connection of Theorem \ref{thm.connection}
can be defined in the same way on any  qc-Einstein manifold with
zero qc-scalar curvature. Other examples can be obtained by considering
$\mathbb{R}^3$-bundles over hyper-K\"ahler manifolds \cite[Remark 5.2]{Ivanov&V10}.
\subsection{The $7$-dimensional quaternionic Heisenberg group, $G_2$-geometry, 
and its spinorial properties}\label{sec.n=7}
%
Let us explain how the connection $\nabla$ appears naturally as the 
characteristic connection
of a $G_2$ structure in dimension 7, i.\,e.~when $p=1$. Consider the three-form
$$
\omega = -\eta_1\wedge (\theta_{12} + \theta_{34})- \eta_2 \wedge 
(\theta_{13}+\theta_{42})- \eta_3  \wedge (\theta_{14}+\theta_{23}) + \eta_{123}.
$$
This is a globally defined three-form of generic type and therefore equips our 
$7$-dimensional Lie group $N_1$ with a $G_2$ structure. Notice that
$$
T =\lambda(\omega - 5 \eta_{123}).
$$ 
It is easy to check that this structure is cocalibrated, that is, $d^\ast\omega =0$. 
Therefore, it is known
that our manifold has characteristic connection $\nabla$ with torsion \cite{FI02},
$$
T^c = \frac{1}{6} (d\omega, \ast \omega)\omega - \ast d\omega.
$$
It is a simple calculation to check that $T = T^c$, but observe also that this torsion 
is extremely similar in spirit to the one given for the 3-Sasaki structure in \cite{AF10}.
Being globally diffeomorphic to $\mathbb{R}^7$, the quaternionic Heisenberg 
group $N_1$ carries
a unique left-invariant spin structure. As a matter of fact, the $7$-dimensional spin 
representation is real, so denote by $\Sigma$ the $8$-dimensional real spinor bundle 
over $N_1$. As a $G_2$-manifold, it has a $\nabla$-parallel spinor field $\psi_0$ that 
may be used to split $\Sigma$ into three summands,
$$
\Sigma \ =\ \Sigma_0 \oplus \Sigma_v \oplus \Sigma_h, \quad
\Sigma_0 \ :=\ \mathbb{R}\cdot\psi_0, \quad
\Sigma_{v,h} \ :=\ \{ X\cdot\psi_0 \, :\, X\in T^{v,h}\}.
$$
The following identities are tensorial, they can therefore be checked in
any realisation of the real spin representation. In one such representation,
one views $T$ as an endomorphism of the spin bundle (replacing all wedge products 
by Clifford products) and computes, in this basis, an explicit formula for $\psi_0$
(it is the unique spinor preserved by $T$).
An explicit purely algebraic computer calculation yields then the following
result:
\begin{lemma}
The spinor field $\psi_0$ is a $T$-eigenspinor,
T$\cdot \psi_0 = -2\lambda\, \psi_0$, and the Clifford product $T\cdot X$
acts on $\psi_0$ as follows:
$$
T\cdot X\cdot \psi_0 = 6 \lambda\, X\cdot \psi_0\ \text{ for }X\in T^h,\quad
T\cdot X\cdot \psi_0 = -4 \lambda\, X\cdot \psi_0\text{ for }X\in T^v.
$$
\end{lemma}
Thus, $T$ acts as multiplication on each of the subbundles of $\Sigma$. 
The spinor field $\psi_0$ has to be a generalized Killing spinor because
the $G_2$ structure is cocalibrated \cite{CS06}. The preceding 
lemma together with $\nabla \psi_0=0$ allows us to compute the 
explicit differential equation of $\psi_0$:
\begin{corollary}\label{cor.gen.KS-psi}
The spinor field $\psi_0$ is a generalized Killing spinor satisfying
the differential equation 
$$
\nabla^g_X \psi_0 = \frac{\lambda}{2}\,X\cdot \psi_0\ \text{ for } X\in T^v,\quad
\nabla^g_X \psi_0 = -\frac{3\lambda}{4}\,X\cdot \psi_0\ \text{ for } X\in T^h.
$$
\end{corollary}
In the $3$-Sasaki case, the spinor fields $\xi_i\cdot \psi_0$ are  exactly the
Riemannian Killing spinors, and they define a nearly parallel $G_2$ structure. 
This cannot hold on $N_1$ (being a nilpotent Lie group, it cannot carry an Einstein metric), so 
it becomes an interesting question
to compute instead the field equation that these three spinors satisfy.
We prove:
\begin{corollary}
The spinor fields $\psi_i:=\xi_i\cdot \psi_0$, $i=1,2,3$, are generalized Killing spinors
satisfying the differential equation 
$$
\nabla^g_{\xi_i} \psi_i = \frac{\lambda}{2}\,\xi_i\cdot \psi_i,\quad
\nabla^g_{\xi_j} \psi_i = - \frac{\lambda}{2}\,\xi_j\cdot \psi_i \ (i\neq j),\quad
\nabla^g_X \psi_i = \frac{5\lambda}{4}\,X\cdot \psi_i\ \text{ for } X\in T^h.
$$
\end{corollary}
\begin{proof}
Denote by $s(X)$ the eigenvalue such that $\nabla^g_X \psi_0= s(X) X\cdot \psi_0$
as in Corollary $\ref{cor.gen.KS-psi}$, i.\,e.~$s(X)=\frac{\lambda}{2} $
resp.~$s(X)=-\frac{3\lambda}{4} $ for $X$ in $T^v$ resp.~$T^h$. 
As duals of Killing vector fields, the one-forms $\eta_i$ satisfy the equation 
$\nabla^g_X \eta_i= \frac{1}{2} X\haken d\eta_i$, hence
\begin{equation}\label{eqq}
\nabla^g_X (\xi_i\cdot \psi_0) = \nabla^g_X (\xi_i)\cdot \psi_0 +
\xi_i \cdot \nabla^g_X \psi_0 =
\frac{1}{2} (X\haken d\eta_i)\cdot \psi_0 + s(X)\, \xi_i \cdot X\cdot\psi_0.
\end{equation}
Consider first the case that $X\in T^v$. In this situation, the first term vanishes because of 
the expressions for $d\eta_i$, see Eq. (\ref{eq.d-eta2}). 
Furthermore, $s(X)=\lambda /2$, so we obtain
$$
\nabla^g_X (\xi_i\cdot \psi_0) =  \frac{\lambda}{2} \xi_i\cdot X\cdot\psi_0.
$$
If $X=\xi_i$, the first statement of the corollary follows, and
for $X=\xi_j$ with $i\neq j$,  the second statement follows after
use of the identity $\xi_i\cdot \xi_j=-\xi_j\cdot \xi_i$.

Now assume $X\in T^h$. We can set $s(X)=-3\lambda /4$,
and a computer computation in the real spin representation 
proves $(X\haken d\eta_i)\cdot \psi_0= X\cdot\xi_i\cdot \psi_0$. Therefore,
Eq. (\ref{eqq}) becomes 
$$
\nabla^g_X (\xi_i\cdot \psi_0) = \frac{\lambda}{2}X\xi_i\cdot \psi_0 - 
s(X)\, X\cdot \xi_i \cdot\psi_0 =  \frac{5\lambda}{4} X\cdot\xi_i\cdot\psi_0
$$
(observe that $\xi_i\notin T^h$, so 
$X\cdot \xi_i= -\xi_i\cdot X$ for all possible $X$).
\end{proof}
\begin{remark}
Generalized Killing spinors with three distinct eigenvalues seem to be very rare---the
authors do not know of any other examples. Moroianu and Semmelmann investigated the 
existence  of generalized Killing spinors with two distinct eigenvalues on spheres $S^n$ 
and proved that they only exist if $n = 3$ or $n = 7$ \cite{Moroianu&S14}; in the latter 
case, they are induced from the canonical spinor of the underlying 3-Sasaki structure 
introduced in \cite{AF10}.
\end{remark}
The first author and J.\,H\"oll discussed cones of $G$ manifolds and their spinorial 
properties in the article \cite{Agricola&H14}. In Section 3.5, 
they constructed three almost Hermitian structures on the cone of a
metric almost contact $3$-structure. Using this construction, we prove:
\begin{theorem}
The cone $(\bar{N}_1, \bar g) = (N_1\times \mathbb{R}^+, \lambda^2r^2 g+dr^2)$ 
is an $8$-dimensional
hyper-K\"ahler manifold with torsion (`HKT manifold').
\end{theorem}
\begin{proof}
Let $T_i$ be the torsion of the characteristic connection of the normal
almost contact structures $(\varphi_i,\xi_i) \ (i=1,2,3)$ of $N_1$. One checks that
$$
T_i = \eta_i\wedge d\eta_i - \sum_{j=1, j\neq i}^3 \eta_j\wedge d\eta_j.
$$
We apply Theorem 3.23 of \cite{Agricola&H14}. 
The crucial point is that we have to
show the existence of a positive constant $a$ (the cone constant) such that 
the three tensors $S_i:= T_i - 2a\,\eta_i\wedge F_i$ coincide. This happens exactly when $a$
equals the metric parameter $\lambda$, and the claim follows.
\end{proof}
This generalizes of course the well-known fact that the cone of a $3$-Sasaki manifold
is a hyper-K\"ahler manifold.
\begin{remark}
Corollary 4.14 of the same article \cite{Agricola&H14} yields immediatly that the
connection $\nabla$ on $N_1$, viewed as the characteristic connection of 
a cocalibrated $G_2$-structure, induces a $\mathrm{Spin}(7)$-structure on the same cone,
and that the spinor $\psi_0$ lifts to a spinor that is parallel for the
characteristic $\mathrm{Spin}(7)$-connection. Thus, the HKT structure is compatible in
a very subtle sense with a $\mathrm{Spin}(7)$-structure.
\end{remark}
%



\begin{thebibliography}{0}
%
\bibitem{ACFH14}
I. Agricola, S.\,G. Chiossi, T. Friedrich, J. H\"oll,
\emph{Spinorial description of $SU(3)$- and $G_2$-manifolds},
preprint, arXiv:/1411.5663[math.DG] 

\bibitem{AFF14} I. Agricola, A.C. Ferreira, T. Friedrich, \emph{The classification of 
naturally reductive homogeneous spaces in dimensions $n\leq 6$}, 
Differ.\,Geom.\,Appl. 39 (2015), 59-92. 

\bibitem{ABBK13}
I. Agricola, J. Becker-Bender, H. Kim,
\emph{Twistorial eigenvalue estimates for generalized Dirac operators with torsion},
Adv.\,Math. 243 (2013), 296-329. 

\bibitem{AF10} I. Agricola and T. Friedrich, \emph{3-Sasakian manifolds in dimension 7, 
their spinors and $G_2$ structures}, J.\,Geom.\,Phys.  60 (2010), 326-332.

\bibitem{Agricola&H14}
I. Agricola, J. H\"oll,
\emph{Cones of G manifolds and Killing spinors with skew torsion},
to appear in Annali di Matematica Pura ed Applicata. 

\bibitem{BB12}
J. Becker-Bender, \emph{Dirac-Operatoren und Killing-Spinoren mit Torsion},
Ph.D.\,thesis, Philipps-Universit\"at Marburg, 2012.

\bibitem{Biquard00}
O. Biquard,
\emph{M\'etriques d'Einstein asymptotiquement sym\'etriques},
Ast\'erisque 265 (2000).

\bibitem{Blair}
D.\,E. Blair, \emph{Riemannian Geometry of Contact and Symplectic Manifolds},
Progress in Mathematics, Vol.\,203, Birkh\"auser, 2002.

\bibitem{CS06} D.\,Conti, S.M.\,Salamon, \emph{Reduced holonomy, 
hypersurfaces and extensions}, Int.\,J.\,Geom. Methods Mod.\,Phys. 3 (2006), 
899--912.

\bibitem{Duchemin06}
D. Duchemin, \emph{Quaternionic contact structures in dimension $7$},
Ann.\,Inst.\,Fourier (Grenoble) 56 (2006), 851-885.

\bibitem{FI02}
T. Friedrich and S. Ivanov, \emph{Parallel spinors and connections with 
skew-symmetric torsion in string theory},
Asian J.\,Math.  6 (2002), 303-335.


\bibitem{Ivanov&M&V14}
S. Ivanov, I.  Minchev, D. Vassilev,
\emph{Quaternionic contact Einstein structures and the quaternionic contact Yamabe 
problem}, Mem.\,Am.\,Math.\,Soc. 1086 (2014).

%
%

\bibitem{Ivanov&V10}
S. Ivanov,  D. Vassilev,
\emph{Quaternionic contact manifolds with a closed fundamental $4$-form},
Bull.\,London Math.\,Soc. 42  (2010), 1021-1030.

\bibitem{Kobayashi&N2}
S. Kobayashi and K. Nomizu, \emph{Foundations of differential geometry {II}},
Wiley Classics Library, Wiley Inc., Princeton, 1969, 1996.

\bibitem{Moroianu&S14}
A. Moroianu, U. Semmelmann,
\emph{Generalized Killing spinors on spheres}, 
Ann.\,Glob.\,Anal.\,Geom. 46 (2014), 129-143.

\bibitem{Tricerri93}
F.~Tricerri, \emph{Locally homogeneous Riemannian manifolds},
Rend.\,Sem.\,Mat.\,Univ.\,Politec.\,Torino 50 (1993), 411-426,
Differential Geometry (Turin, 1992).

\bibitem{TV83} 
F. Tricerri and L. Vanhecke, \emph{Homogeneous structures on Riemannian manifolds}, 
London Math.\,Soc.\,Lecture Notes Series, vol. 83, Cambridge Univ. Press, Cambridge, 1983.


%

\end{thebibliography}
\end{document}